\documentclass[11pt,reqno]{amsart}

\usepackage{amsmath,amssymb,mathtools,mathrsfs}
\usepackage{microtype}
\usepackage{enumitem}
\usepackage{booktabs}
\usepackage{hyperref}
\usepackage[nameinlink,capitalise,noabbrev]{cleveref}

\allowdisplaybreaks
\numberwithin{equation}{section}

\newtheorem{theorem}{Theorem}[section]
\newtheorem{lemma}[theorem]{Lemma}
\newtheorem{proposition}[theorem]{Proposition}

\theoremstyle{remark}

\newtheorem*{remark*}{Remark}

\newcommand{\e}{\mathrm e}
\newcommand{\N}{\mathbb N}
\newcommand{\Z}{\mathbb Z}

\newcommand{\ord}{\operatorname{ord}}

\newcommand{\eps}{\varepsilon}
\newcommand{\Sscr}{\mathfrak S}
\newcommand{\calN}{\mathcal N}
\newcommand{\calZ}{\mathcal Z}
\newcommand{\Ssq}{\mathcal S_2}
\DeclareMathAlphabet{\mathbblowercase}{U}{bbold}{m}{n}

\def\pmod #1{\,({\rm mod}\ #1)}

\makeatletter
\@namedef{subjclassname@2020}{\textup{2020} Mathematics Subject Classification}
\makeatother

\title[Linnik's microsquare problem for almost all integers]{The optimal scale in Linnik's microsquare problem for almost all integers}

\author
[Y. Sun and L. Zhao] 
{ Yu-Chen Sun and Lilu Zhao}

\address{(Yu-Chen Sun) School of Mathematics, University of Bristol
Bristol, BS8 1UG, England}
\email{yuchensun93@163.com}
\address{(Lilu Zhao) School of Mathematical Science, University of Science and Technology of China, Hefei 250100, People's Republic of China}
\email{zhaolilu@ustc.edu.cn}

\date{}
\subjclass[2020]{Primary 11E25; Secondary 11P05, 11N37,
11N25.}
\keywords{Sums of three squares, Hooley's weight, Green's W-trick}

\begin{document}
\begin{abstract}
We prove that, for any positive function $\Psi(t)\to\infty$, almost all positive integers $n$ satisfying $4\nmid n$ and $n\not\equiv7\pmod8$ can be written as $n=x^2+y^2+z^2$, where $x,y,z\in\mathbb{Z}$ and $1\leqslant z<\sqrt{\log n}\,\Psi(n)$. This improves Wooley's bound, and we show that the result is best possible. Our approach combines Hooley's weighted representation function for sums of two squares with Green's $W$-trick. We also give a necessary and sufficient condition, in terms of covering systems, for almost all odd integers to be represented as the sum of a squarefree number and a power of two, answering a question of Granville and Soundararajan.
\end{abstract}
\maketitle

\section{Introduction}

The classical theorem on three squares states that a positive integer is a sum of
three integral squares if and only if it is not of the form $4^a(8b+7)$,
where $a,b$ are nonnegative integers.
We call a positive integer $n$ admissible if
\begin{equation}\label{eq:localcongruence}
4\nmid n,\qquad n\not\equiv7\pmod8.
\end{equation}
\emph{Linnik's microsquare conjecture}~\cite{Linnik,Wooley} asks whether, for every $\varepsilon>0$,
every sufficiently large squarefree odd admissible integer has a representation
$n=x^2+y^2+z^2$ with $|z|\leqslant n^\varepsilon$.

The equidistribution results of Duke and Schulze-Pillot~\cite{DukeSchulzePillot}
and Golubeva and Fomenko~\cite{GolubevaFomenkoEquidistribution} imply that,
for every fixed $\eta>0$, every sufficiently large squarefree odd admissible
integer has such a representation with $|z|<\eta\sqrt n$.
Golubeva and Fomenko~\cite{GolubevaFomenko} subsequently proved that,
for every $\varepsilon>0$, every sufficiently large admissible integer $n$ has
such a representation with $|z|\ll_\varepsilon n^{1/2-7/705+\varepsilon}$.
For squarefree $n\equiv3\pmod8$, Humphries and
Radziwi{\l}{\l}~\cite{HumphriesRadziwill} improved this to
$|z|\leqslant n^{4/9+\varepsilon}$, and to $|z|\leqslant n^{1/3+\varepsilon}$
under the generalised Lindel\"of hypothesis.
These bounds hold for every sufficiently large $n$ in the stated ranges.
More recently, assuming the Lindel\"of hypothesis for standard $L$ functions on
$\mathrm{GL}(2)/\mathbb Q$, Burrin and
Gr\"obner~\cite{BurrinGrobner} extended the conditional bound
$|z|\ll_\varepsilon n^{1/3+\varepsilon}$ to all admissible $n$,
with primitive representations.

On the other hand, Wooley~\cite{Wooley} established Linnik's conjecture for
almost all admissible integers.
\begin{theorem}[Wooley \cite{Wooley}]\label{thm-Wooley}
Let $\varepsilon>0$. Almost all positive integers $n$ satisfying
\eqref{eq:localcongruence} can be written as $n=x^2+y^2+z^2$, where
\[
0<z<(\log n)(\log\log n)^{2+\varepsilon},\qquad x,y,z\in\Z.
\]
\end{theorem}

The density of sums of two squares suggests a smaller scale.
Landau~\cite{Landau} proved that
\begin{equation*}
 \#\{m\leqslant x:m=a^2+b^2\text{ for some }a,b\in\Z\}
 \sim A\frac{x}{\sqrt{\log x}},
\end{equation*}
where $A$ is the Landau--Ramanujan constant,
\begin{equation}
 A=\frac1{\sqrt2}\prod_{p\equiv3\pmod{4}}\left(1-\frac1{p^2}\right)^{-1/2}.
 \label{eq:LRconstant}
\end{equation}
Summing over the possible values of $z$ gives
\[
\#\{n\leqslant X:n=x^2+y^2+z^2,\ x,y,z\in\Z,\ 1\leqslant z\leqslant Y\}
\ll \frac{XY}{\sqrt{\log X}}.
\]
Thus, if $Y=o(\sqrt{\log X})$, only $o(X)$ integers can have such a
representation. Our results describe what happens at the critical scale
$\sqrt{\log n}$.

We improve Wooley's bound by proving the following theorem.
\begin{theorem}\label{thm:main}
Let $\Psi:[1,+\infty)\to [1,+\infty)$ satisfy $\Psi(n)\to+\infty$ as $n\to+\infty$. Apart from $o(X)$ exceptions, every integer $n$ satisfying
\[
 1\leqslant n\leqslant X,\qquad 4\nmid n,\qquad n\not\equiv7\pmod8
\]
can be written as $n=x^2+y^2+z^2$, where $0<z<\sqrt{\log n}\,\Psi(n)$ and  $x,y,z\in\Z$.
\end{theorem}

The next theorem shows that Theorem~\ref{thm:main} is best possible.

For $C>0$, $a_0\in\{1,2,3,5,6\}$, and $X\geqslant3$, define
\[
\mathcal E_{C,a_0}(X)=
\left\{n\in\Z:\begin{array}{l}
X<n\leqslant2X,\quad n\equiv a_0\pmod8,\\
n\ne x^2+y^2+z^2\text{ for all }x,y,z\in\Z\\
\text{with }|z|\leqslant C\sqrt{\log n}
\end{array}\right\}.
\]
\begin{theorem}\label{thm:critical-exceptions}
For every fixed $C>0$, there exists $\delta_C>0$ depending only on $C$ such that,
for every $a_0\in\{1,2,3,5,6\}$ and all sufficiently large $X$,
\[
\#\mathcal E_{C,a_0}(X)\geqslant\delta_C X.
\]
\end{theorem}

\begin{remark*}
Work of Richards~\cite{Richards} and Dietmann, Elsholtz, Kalmynin,
Konyagin, and Maynard~\cite{DietmannEtAl}
on gaps between sums of two squares implies that, for every
sufficiently small fixed $C>0$, there are infinitely many positive integers
$n$ satisfying \eqref{eq:localcongruence} that admit no representation
$n=x^2+y^2+z^2$ with $x,y,z\in\Z$ and $|z|\leqslant C\sqrt{\log n}$.
\end{remark*}

Let us briefly discuss the idea of the proof.
The first key ingredient is the W-trick, introduced by Green~\cite{Green}.
In the spirit of localization in harmonic analysis, we partition the
admissible integers into residue classes modulo $W$, estimate the exceptional
proportion in each class, and then let the number of classes tend to infinity
with $W$. The second key ingredient is Hooley's weight~\cite{Hooley}.
The weighted representation count $R_W(\mathbf{n})$ has a second moment
close to the square of its mean~\cite{McGrath,KimmelKuperberg}.
Chebyshev's inequality gives, for every $t>0$,
\[
\mathbb{P}\!\left(\left|R_W(\mathbf{n})-\mathbb{E}R_W(\mathbf{n})\right|
\geqslant t\right)
\leqslant
\frac{\mathbb{E}\!\left(\big(R_W(\mathbf{n})-\mathbb{E}R_W(\mathbf{n})\big)^2\right)}{t^2}.
\]
Taking $t=\mathbb{E}R_W(\mathbf{n})>0$, we obtain
\[
\mathbb{P}\big(R_W(\mathbf{n})=0\big)
\leqslant
\frac{\mathbb{E}\big(R_W(\mathbf{n})^2\big)}
{\big(\mathbb{E}R_W(\mathbf{n})\big)^2}-1
\longrightarrow0\quad\text{as }W\to\infty,
\]
where $\mathbf{n}$ is chosen uniformly from the progression.


The argument for Theorem~\ref{thm:critical-exceptions} is different.
We use the distribution of sums of two squares in fixed arithmetic
progressions, based on a theorem of Prachar~\cite{Prachar}.
For a given $C$, we choose a progression whose averaged local factors make
the average number of shifts $n-z^2$ belonging to the set of sums of two
squares small. A first moment bound over
the permitted values of $z$ then leaves a positive proportion of integers
in that progression without a representation.

As a complementary result, we give one more application of Green's $W$-trick by considering representations of odd integers as
the sum of a squarefree number and a power of $2$. Motivated by a question
of Granville and Soundararajan~\cite{GS}, Theorem~\ref{thm:squarefree-covering}
characterizes the representation property for almost all odd integers in terms of covering
systems whose moduli are the multiplicative orders of $2$ modulo squares
of odd primes. Its proof also separates finite local obstructions before
estimating the remaining contribution.

The rest of the paper is organized as follows.
Section~\ref{sec:notation} introduces the weights and deduces
Theorem~\ref{thm:main} from Proposition~\ref{lem:variance}.
The required moment estimates are developed in Section~\ref{sec:weighted-moments},
and the proof of the proposition is completed in Section~\ref{sec:branch-moments}.
Section~\ref{sec:critical-exceptions} proves Theorem~\ref{thm:critical-exceptions}.
The final section contains the result on squarefree numbers.

\section{Notation and proof of Theorem~\ref{thm:main}}\label{sec:notation}

All notations introduced in this section will be used throughout this paper. Let $\eps>0$ be a sufficiently small constant, say $\eps<10^{-10}$. Let $X>0$ be sufficiently lage. The letter $p$ denotes a prime. We write
\[
 e(t)=\e^{2\pi i t},\qquad \chi=\chi_4,
\]
where $\chi$ is the nonprincipal real character modulo $4$. 

Fix $j_0\in \{1,2,3\}$, and let 
\begin{align}\label{eq:definej0lambda0}
\widetilde{j_0}=\begin{cases}0, \ \  & \ j_0\in\{1\}
\\ 1, \ \  &\ j_0\in\{2,3\}\end{cases}\ \  \textrm{and}\ \ \
\lambda_0=\begin{cases}1, \ \  & \ j_0\in\{1,2\}
\\ 2, \  \ &\ j_0=3\end{cases}.\end{align}

Let $w\geqslant 3$ be sufficiently large and put
\begin{equation}
 W=W(w)=\prod_{3\leqslant p\leqslant w}p.
 \label{eq:W-product}
\end{equation}

Write
\begin{equation}
 W=W_1W_3,
 \label{eq:W13}
\end{equation}
where every prime dividing $W_j$ is congruent to $j$ modulo $4$.

\begin{lemma}\label{lem:local-square}
For every residue class $c\pmod W$, there exists $\beta_c\pmod W$ such that
\begin{equation*}
 (c-\beta_c^2,W)=1.
\end{equation*}
\end{lemma}

\begin{proof}
For a fixed odd prime $p$, the congruence $b^2\equiv c\pmod p$ has at most two solutions. Hence at least one residue class $b\pmod p$ avoids it. Choose such a class for every $p\mid W$ and combine the choices by the Chinese remainder theorem.
\end{proof}

From now on we fix an arbitrary residue class $c\pmod W$, and choose $\beta_c\pmod W$ satisfying Lemma \ref{lem:local-square}.
Then we define $\mathcal{N}:=\mathcal{N}(X;c,W;j_0)$ and $\mathcal{Z}:=\mathcal{Z}(X;c,W;j_0)$ by 
\begin{align}\label{definecalN}
\mathcal{N}=\{X<n\leqslant 2X:\ n\equiv c\pmod{W},\ n\equiv j_0\pmod{4\lambda_0}\}
\end{align}
and
\begin{align}\label{eq:definecalZ}
\mathcal{Z}=\{1\leqslant z\leqslant Y:\ z\equiv \beta_c\pmod{W},\ z\equiv\widetilde{j_0}\pmod{2}\},\end{align}
where 
\begin{align}\label{sizeY} \sqrt{\log X}<Y<\log X.\end{align}
It is easy to see that
\begin{align}\label{sizeNandZ}
|\mathcal{N}|=\frac{X}{4\lambda_0W}+O(1) \quad \textrm{and} \quad 
|\mathcal{Z}|=\frac{Y}{2W}+O(1).
\end{align}
For $n\in \mathcal{N}$ and $z\in \mathcal{Z}$, we define 
\begin{equation}
 m_{n,z}=\frac{n-z^2}{\lambda_0}.
 \label{eq:mz}
\end{equation}

\begin{lemma}\label{lem:three-branches}Let $n\in \mathcal{N}$ and $z\in \mathcal{Z}$, and let $m_{n,z}$ be defined as in \eqref{eq:mz}. Then
\begin{equation*}
m_{n,z}\in\Z^{+},
 \qquad m_{n,z}\equiv1\pmod4,
 \qquad (m_{n,z},W)=1.
\end{equation*}
Moreover, if $m_{n,z}=a^2+b^2$, then $n=x^2+y^2+z^2$ for suitable integers $x,y$.
\end{lemma}

\begin{proof} The proof is easy. We only discuss the case $j_0=3$. Note that $\lambda_0=2$ and $\widetilde{j_0}=1$ in this case. Since $n-z^2\equiv j_0-1\equiv 2\pmod{8}$, we have $m_{n,z}\in \Z$ and $m_{n,z}\equiv 1\pmod{4}$. The conclusion $(m_{n,z},W)=1$ follows from Lemma \ref{lem:local-square}. 
We deduce from $m_{n,z}=a^2+b^2$ that $n=(a-b)^2+(a+b)^2+z^2$. Finally, by \eqref{sizeY}, $n-z^2>0$ and thus $m_{n,z}\in \Z^{+}$.
\end{proof}
For $m\geqslant 1$, let
\[
 r_2(m)=\#\{(a,b)\in\Z^2:a^2+b^2=m\}.
\]
Put
\begin{equation*}
 v=X^\theta \qquad \textrm{ with }\quad \theta=\frac{1}{20}.
\end{equation*}
Let $\mu(a)$ denote the M\"obius function. 
Define the multiplicative functions $g_1$ and $g_2$ on primes by
\begin{equation}
 g_1(p)=1-\frac{\chi(p)}p,
 \qquad
 g_2(p)=
 \begin{cases}
  2-p^{-1},&p\equiv1\pmod4,\\
  p^{-1},&p\equiv3\pmod4.
 \end{cases}
 \label{eq:g1g2}
\end{equation}
Only at odd squarefree values of $g_1$ and $g_2$ will occur. Now we introduce Hooley's weight
\begin{equation}
 t_v(m)=
 \sum_{\substack{a\leqslant v\\ a\mid m \\p\mid a\Rightarrow p\equiv1\, (4)}}
 \frac{\mu(a)}{g_2(a)}\frac{\log(v/a)}{\log v}
 \label{eq:tv}
\end{equation}
and
\begin{equation}
 \rho_v(m)=r_2(m)t_v(m)\qquad(m\geqslant 1).
 \label{eq:rhov}
\end{equation}
The weight $t_v$ need not be nonnegative. We never use nonnegativity. What matters is the support implication
\begin{equation}
 m\text{ is not a sum of two squares}\quad\Longrightarrow\quad \rho_v(m)=0.
 \label{eq:support-implication}
\end{equation}
Let 
\begin{equation}
 R(n)=\sum_{z\in\calZ}\rho_v(m_{n,z})
 =\sum_{z\in\calZ}\rho_v\!\left(\frac{n-z^2}{\lambda_0}\right).
 \label{eq:Rn}
\end{equation}

We can now state our main proposition.
\begin{proposition}
  \label{lem:variance}
Let $W=W(w)$, $R$ and $\calN$ be as in \eqref{eq:W-product},
\eqref{eq:Rn} and \eqref{definecalN}, respectively.
Suppose that \eqref{sizeY} holds and put
\[
 \mu_0=\frac1{|\calN|}\sum_{n\in\calN}R(n).
\]
Then, for fixed $W$ and all sufficiently large $X$, one has $\mu_0>0$ and
\begin{equation}
 \sum_{n\in\calN}|R(n)-\mu_0|^2
 \leqslant \Big(O_W(1)\frac{\sqrt{\log X}}{Y}
 +O(w^{-1})+o_W(1)\Big)|\calN|\mu_0^2.
 \label{eq:branch-variance}
\end{equation}
The implied constant in $O(w^{-1})$ is absolute, and the error terms are
uniform in the admissible choices of
$Y$, $j_0$, $c$ and $\beta_c$.
\end{proposition}

The proof of Proposition \ref{lem:variance} is deferred to Section~\ref{sec:branch-moments}.

\begin{proof}[Proof of Theorem \ref{thm:main}]
We first reduce an arbitrary function $\Psi(n)\to\infty$ to a slowly growing dyadic parameter. For $X$ large, define
\begin{equation*}
 \Psi_*(X)=\min_{X<n\le2X}\Psi(n),
 \qquad
 \psi(X)=\min\{\Psi_*(X),\log\log\log X\}.
\end{equation*}
Then $\psi(X)\to\infty$. Put
\begin{equation*}
 Y=\left\lfloor\sqrt{\log X}\,\psi(X)\right\rfloor.
\end{equation*}
For every $n\in(X,2X]$ and all sufficiently large $X$,
\begin{equation}
 \frac{1}{2}\sqrt{\log X}\,\psi(X)<Y<\sqrt{\log n}\,\Psi(n),
 \label{eq:Y-below-target}
\end{equation}
so \eqref{sizeY} holds.

Fix $w$ and hence $W=W(w)$. Partition the integers $n\in(X,2X]$ satisfying \eqref{eq:localcongruence} into the finitely many branches $\mathcal{N}$ defined in \eqref{definecalN}. On every branch,
\[
 |\mathcal{Z}|=\frac{Y}{2W}+O(1)\gg _W\sqrt{\log X}\,\psi(X).
\]
For sufficiently large $X$, Proposition \ref{lem:variance} gives $\mu_0>0$.
Since $R(n)=0$ implies $|R(n)-\mu_0|^2=\mu_0^2$, Proposition \ref{lem:variance} yields
\begin{align*}
 \frac{\#\{n\in\calN:R(n)=0\}}{|\calN|}
 &\leqslant \frac{1}{|\calN|\mu_0^2}\sum_{n\in\calN}|R(n)-\mu_0|^2\\
 &\leqslant O_W(1)\frac1{\psi(X)} +O(w^{-1})+o_W(1).
\end{align*}
Since $\frac1{\psi(X)}=o(1)$, it follows that
\begin{equation*}
\frac{\#\{n\in\calN:R(n)=0\}}{|\calN|}
 \leqslant O(w^{-1})+o_W(1).
\end{equation*}
The number of branches depends only on $W$, so the same bound holds after summing over all branches. Thus,
\begin{equation*}
\frac{\#\{n\in (X,2X]:\ R(n)=0 \ \textrm{and}\ n \ \textrm{satisfies}\ \eqref{eq:localcongruence}\}}{
\#\{n\in (X,2X]:\ n \ \textrm{satisfies}\ \eqref{eq:localcongruence}\}}
 \leqslant O(w^{-1})+o_W(1).
\end{equation*}

For any $\eps>0$, choose $w=w_\eps$ to be sufficiently large such that $O(w^{-1})$ term is $<\eps/4$. Then choose $X_0(\eps)$ sufficiently large such that whenever $X\geqslant X_0(\eps)$ then the $o_W(1)$ term is also $<\eps/4$. Then obtain that \begin{equation*}
\frac{\#\{n\in (X,2X]:\ R(n)=0 \ \textrm{and}\ n \ \textrm{satisfies}\ \eqref{eq:localcongruence}\}}{
\#\{n\in (X,2X]:\ n \ \textrm{satisfies}\ \eqref{eq:localcongruence}\}}
 <\eps/2
\end{equation*}
for $X>X_0(\eps)$. We introduce $\mathcal{E}(X_1,X_2)$ to denote the set consisting of $n\in (X_1,X_2]$ satisfying \eqref{eq:localcongruence} and that $n$ cannot be represented in the form $n=x^2+y^2+z^2$ with $0<z<\sqrt{n}\Psi(n)$. By Lemma \ref{lem:three-branches}, \eqref{eq:support-implication} and \eqref{eq:Y-below-target}, 
$$\mathcal{E}(X,2X)\subseteq \{n\in (X,2X]:\ R(n)=0 \ \textrm{and}\ n \ \textrm{satisfies}\ \eqref{eq:localcongruence}\}.$$
Therefore, 
for $X>X_0(\eps)$,
\begin{equation*}
\#\mathcal{E}(X,2X)<
\frac{\eps}{2}\#\{n\in (X,2X]:\ n \ \textrm{satisfies}\ \eqref{eq:localcongruence}\}
\end{equation*}

Then a dyadic argument gives
\begin{equation*}
\#\mathcal{E}(X_0(\eps),X)<\frac{\eps}{2}\big(X-X_0(\eps)\big)<\frac{\eps}{2}X
\end{equation*}
and for $X>\frac{2}{\eps}X_0(\eps)$,
\begin{equation*}
\#\mathcal{E}(0,X)<\eps X.
\end{equation*}
This proves Theorem \ref{thm:main}.
\end{proof}

\section{Moments of $\rho_v(m)$ in arithmetic progressions}\label{sec:weighted-moments}

We first introduce some elementary estimates on the Ramanujan sum.
For $q\geqslant 1$, the Ramanujan sum is
\begin{equation}
 c_q(h)=\sum_{\substack{a\bmod q\\(a,q)=1}}e\!\left(\frac{ah}{q}\right)
 =\sum_{d\mid(q,h)}d\,\mu(q/d).
 \label{eq:Ramanujan}
\end{equation}
It is real and even in $h$. We write
\[
 \tau(m)=\sum_{d\mid m}1,
 \qquad
 \sigma_s(m)=\sum_{d\mid m}d^s.
\]

\begin{lemma}Suppose that $h\not=0$. Then
  \begin{equation}
 \sum_{q>T}\frac{|c_q(h)|}{q^2}\ll\frac{\tau(|h|)}{T}
 \label{eq:Ramanujan-bounds}
\end{equation}
and for $1<k\leqslant 2$ \begin{equation}
 \sum_{q=1}^\infty\frac{|c_q(h)|}{q^{k}}\ll  \sigma_{1-k}(|h|).
 \label{eq:Ramanujan-bound2}
\end{equation}
\end{lemma}
\begin{proof}One has\begin{align*}
 \sum_{q>T}\frac{|c_q(h)|}{q^{k}}=\sum_{d\mid h}d\sum_{\substack{q>T \\ d\mid q }}\frac{\mu(q/d)}{q^k}
 \ll \sum_{d\mid h}d^{1-k}\sum_{t>T/d}\frac{1}{t^k} .
\end{align*}
When $k=2$, we use $\sum_{t>T/d}\frac{1}{t^2}\ll d/T$ to obtain \eqref{eq:Ramanujan-bounds}. When $T=1$, we use $\sum_{t>1/d}\frac{1}{t^2}\ll 1$ to obtain \eqref{eq:Ramanujan-bound2}.
\end{proof}

\subsection{The cited $r_2$ estimates}

In this subsection, we quote some known results to make the logical dependence of the proof transparent. For convenience, we often simplify the notation $x\equiv a\pmod{q}$ to be $x\equiv a(q)$. We may recall the definitions of $g_1$ and $g_2$ in \eqref{eq:g1g2}.

\begin{lemma}\label{lemma:first-r2}
Let $W,d$ be odd squarefree integers with $(d,W)=1$, and let $(\alpha,W)=1$. Uniformly for $x\asymp X$,
\begin{equation*}
 \sum_{\substack{m\leqslant x\\m\equiv\alpha\, (W)\\m\equiv1\, (4)\\d\mid m}}
 r_2(m)
 =\frac{\pi x}{2Wd}g_1(W)g_2(d)+R_1(x;d,W),
\end{equation*}
where
\begin{equation*}
 R_1(x;d,W)\ll_\eps W^{1/2}dx^{\eps}+x^{1/3+\eps}d^{1/2}.
\end{equation*}
\end{lemma}
This is McGrath~\cite[Lemma 5.3]{McGrath}, see also Kimmel--Kuperberg~\cite[Lemma 10]{KimmelKuperberg}.

\begin{lemma}\label{lem:shifted-r2}
Let $W,c_1,c_2$ be odd squarefree integers such that
\[
 (\alpha,W)=(\alpha+h,W)=(c_1,W)=(c_2,W)=1,
 \qquad 4\mid h.
\]
Assume also that $0<h<x^{3/4}$. Then
\begin{align*}
 &\sum_{\substack{m\leqslant x\\m\equiv\alpha\, (W)\\m\equiv1\, (4)\\c_1\mid m\\c_2\mid m+h}}
 r_2(m)r_2(m+h)\notag\\
 &\qquad =\frac{\pi^2x}{W}g_1(W)^2\Gamma(h,c_1,c_2,W)
 +R_2(x;c_1,c_2,W),
\end{align*}
where
\begin{equation*}
 \Gamma(h,c_1,c_2,W)
 =\frac{g_2(c_1)g_2(c_2)}{c_1c_2}
 \sum_{\substack{t\geqslant 1\\(t,2W)=1}}
 \frac{c_t(h)(c_1,t)(c_2,t)
 \chi((c_1^2,t))\chi((c_2^2,t))}
 {t^2\Psi(c_1,t)\Psi(c_2,t)}
\end{equation*}
with
\begin{equation}
 \Psi(u,t)=g_2\!\left(\Big(u,\frac{t}{(u,t)}\Big)\right),
 \label{eq:Psi-local}
\end{equation}
and
\begin{equation*}
 R_2(x;c_1,c_2,W)
 \ll_\eps W^{1/2}c_1c_2x^{3/4+\eps}
 +(c_1c_2)^{1/2}x^{5/6+\eps}.
\end{equation*}
If $(c_1,c_2)\nmid h$, both the original sum and the main term vanish.
\end{lemma}
This is Kimmel--Kuperberg \cite[Lemma 11]{KimmelKuperberg}, see also McGrath~\cite[Lemma A.3]{McGrath}.

\begin{lemma}\label{lemma:second-r2}
Let $W,d$ be odd squarefree integers with $(d,W)=1$, and let $(\alpha,W)=1$. There exists functions $g_3,g_4,G_5,G_6$ such that uniformly for $x\asymp X$,
\begin{equation*}
 \sum_{\substack{m\leqslant x \\ m\equiv\alpha\, (W)\\m\equiv1\, (4)\\ d\mid m}}
 r_2(m)^2
 =\frac{g_3(W)g_4(d)}{Wd} \Big(\log x+G_5(W)-2G_6(d)\Big)x
 +O_{\eps}(Wx^{\frac{3}{4}+\eps}).
\end{equation*}
\end{lemma}
This is McGrath \cite[Lemma 5.5]{McGrath}, see also Kimmel--Kuperberg \cite[Lemma 12]{KimmelKuperberg}. Note that we do not need the precise definition of $g_3,g_4,G_5,G_6$ in the latter proof.

\begin{lemma}\label{lemma:XZonetwo}Suppose that if $p\mid Q$ then $p\equiv 1\pmod{4}$. 

(i) Put
\begin{align*}
 X_{v,Q}=&\,
 \sum_{\substack{a\le v\\(a,Q)=1\\p\mid a\Rightarrow p\equiv1\, (4)}}
 \frac{\mu(a)}a\log\frac va.
 \end{align*}Then
\begin{equation}
 X_{v,Q}=(1+o_Q(1))\frac{8A}{\pi g_1(Q)}\sqrt{\log v}.
 \label{eq:XvQ-asymp}
\end{equation}
 
 (ii) Let $g_4,G_6$ be the functions appearing in Lemma \ref{lemma:second-r2}. Put
 \begin{align*}
  Z^{(1)}_{v,Q}=&\,
 \sum_{\substack{a,b\leqslant v\\(ab,Q)=1\\p\mid ab\Rightarrow p\equiv1\, (4)}}
\frac{\mu(a)\mu(b)g_4([a,b])}{g_2(a)g_2(b)[a,b]}\big(\log \frac{v}{a}\big)\big(\log \frac{v}{b}\big), \\
 Z^{(2)}_{v,Q}=&\,
 \sum_{\substack{a,b\leqslant v\\(ab,Q)=1\\p\mid ab\Rightarrow p\equiv1\, (4)}}
\frac{\mu(a)\mu(b)g_4([a,b])}{g_2(a)g_2(b)[a,b]}\big(\log \frac{v}{a}\big)\big(\log \frac{v}{b}\big)G_6([a,b]).
\end{align*}
Then there exists functions $g_7,g_8$ such that
\begin{align}\label{eq:Zonetwo}Z^{(1)}_{v,Q}=(1+o(1))g_7(W)(\log v)^{1/2}, \qquad Z^{(2)}_{v,W}=-(1+o(1))g_8(W)(\log v)^{3/2}.
\end{align}
\end{lemma}
This is McGrath \cite[Lemma 5.6]{McGrath}, see also Kimmel--Kuperberg \cite[Lemma 13]{KimmelKuperberg}.

\begin{lemma}\label{lemma:DvQt}Suppose that $(t,Q)=1$. Suppose further that all primes dividing $tQ$ are congruent to $1$ modulo $4$.  Recall \eqref{eq:Psi-local}. Define
\begin{equation}
 D_{v,Q}(t)=
 \sum_{\substack{a\leqslant v\\(a,Q)=1\\p\mid a\Rightarrow p\equiv1\, (4)}}
 \frac{\mu(a)(a,t)}{a\Psi(a,t)}\log\frac va.
 \label{eq:DvQ}
\end{equation}Then
\begin{equation*}
 D_{v,Q}(t)=\frac{8A}{\pi g_1(Q)}C(t)\sqrt{\log v}
 +O_Q\!\left(\frac{t^{1/4}}{\sqrt{\log v}}\right),
\end{equation*}
where
\begin{equation}
 C(t)=
 \begin{cases}
  \displaystyle\prod_{p\,\mid\, t}\left(2-\frac1p\right)^{-1},
  &\text{if }p\mid t\Rightarrow p^2\mid t,\\[3mm]
  0,&\text{otherwise.}
 \end{cases}
 \label{eq:Ct}
\end{equation}
\end{lemma}
\begin{proof}
This estimate was proved by Hooley \cite[Lemma 6]{Hooley} and is also stated in
\cite[p.~2027]{KimmelKuperberg}. As noted in \cite{KimmelKuperberg}, Hooley's
statement contains a typographical error in the main term constant.
We only need an upper bound for $D_{v,Q}(t)$ in the subsequent argument.
\end{proof}

\subsection{The estimates of $\rho_v(m)$} In this subsection, we deduce moments of $\rho_v(m)$ in arithmetic progressions based on the quoted results in Section 3.1.

\begin{lemma}\label{lemma-firstmoment}Let $(\alpha,W)=1$. Uniformly for intervals $I\subset[X/3,3X]$ of length $\asymp X$, the following hold as $X\to\infty$.
Then
 \begin{equation}
  \sum_{\substack{m\in I\\m\equiv\alpha\, (W)\\m\equiv1\, (4)}}
  \rho_v(m)
  =\frac{4Ag_1(W_3)|I|}{W\sqrt{\log v}}+o_W\Big(\frac{X}{\sqrt{\log X}}\Big).
  \label{eq:rho-first}
 \end{equation}
\end{lemma}
\begin{proof}Recalling \eqref{eq:rhov} and expanding $\rho_v$ using \eqref{eq:tv}, we obtain
\[
  \sum_{\substack{m\in I\\m\equiv\alpha\, (W)\\m\equiv1\, (4)}}
  \rho_v(m)=\frac{1}{\log v}
 \sum_{\substack{a\leqslant v\\(a,W)=1\\p\mid a\Rightarrow p\equiv1\, (4)}}
 \frac{\mu(a)\log (v/a)}{g_2(a)}
 \sum_{\substack{m\in I\\m\equiv\alpha\, (W)\\m\equiv1\, (4)\\a\mid m}}
 r_2(m).
\]
The condition $(a,W)=1$ was introduced in the above summation, because the inner sum is zero when $(a,W)>1$. Now using Lemma \ref{lemma:first-r2} with $d=a$, we conclude that
\[
  \sum_{\substack{m\in I\\m\equiv\alpha\, (W)\\m\equiv1\, (4)}}
  \rho_v(m)= \mathcal{M}_1+O(\mathcal{E}_1),
\]
where
\begin{equation}\label{calM1}
 \mathcal{M}_1= \frac{\pi|I|g_1(W)}{2W\log v}
 \sum_{\substack{a\leqslant v\\(a,W)=1\\p\mid a\Rightarrow p\equiv1\, (4)}}
  \frac{\mu(a)\log (v/a)}{a}
\end{equation}
and
\[
  \mathcal{E}_1= \frac{1}{\log v}
 \sum_{\substack{a\leqslant v\\(a,W)=1\\p\mid a\Rightarrow p\equiv1\, (4)}}
 \frac{|\mu(a)|\log (v/a)}{g_2(a)}\Big(W^{1/2}aX^{\eps}+X^{1/3+\eps}a^{1/2}\Big).
\]
Note that the inner sum in \eqref{calM1} is $X_{v,W_1}$, because $(a,W)=1$ is equivalent to $(a,W_1)=1$. Inserting \eqref{eq:XvQ-asymp} and using $g_1(W)=g_1(W_1)g_1(W_3)$ gives 
\begin{equation*}
   \mathcal{M}_1=\frac{4Ag_1(W_3)|I|}{W\sqrt{\log v}}\left(1+o_W(1)\right)=\frac{4Ag_1(W_3)|I|}{W\sqrt{\log v}}+o_W\Big(\frac{X}{\sqrt{\log X}}\Big).
\end{equation*}
For $ \mathcal{E}_1$, we deduce that
\[
  \mathcal{E}_1 \ll_{W,\eps}X^\eps\left(\sum_{a\leqslant v}a+X^{1/3}\sum_{a\leqslant v}a^{1/2}\right)
 \ll_{W,\eps}X^\eps(v^2+X^{1/3}v^{3/2}),
\]
which is $o_W(X/\sqrt{\log X})$. This completes the proof.
\end{proof}

We now turn to the shift sum. We define a multiplicative function $\gamma$ determined by
\begin{equation}
 \gamma(p^k)=
 \begin{cases}
  1,&p\equiv3\pmod4,\ k\geqslant 1,\\
  0,&p\equiv1\pmod4,\ k=1,\\
  (2-p^{-1})^{-2},&p\equiv1\pmod4,\ k\geqslant 2.
 \end{cases}
 \label{eq:gamma}
 \end{equation}
and introduce \begin{align}
 M_W
 &=\sum_{\substack{t\geqslant 1\\(t,2W)=1}}\frac{\gamma(t)\phi(t)}{t^3},
 \label{eq:MW}
\end{align}
where $\phi(t)$ is the Euler function.
Recalling \eqref{eq:W-product}, we have
\begin{equation}
 M_W=1+O(w^{-1}).
 \label{eq:MW-tail}
\end{equation}
Indeed, $0\leqslant\gamma(t)\leqslant1$ for odd $t$, and the terms with $t>1$
in \eqref{eq:MW} satisfy $t>w$. Thus
$0\leqslant M_W-1\leqslant\sum_{t>w}t^{-2}\ll w^{-1}$,
with an absolute implied constant.

Recalling \eqref{eq:Ramanujan}, for $h\ne0$ we define
\begin{equation*}
 \Sscr_W(h)=
 \sum_{\substack{t\geqslant 1\\(t,2W)=1}}
 \frac{\gamma(t)c_t(h)}{t^2}.
\end{equation*}
By \eqref{eq:Ramanujan-bounds}, this series converges absolutely for every nonzero $h$. Recall \eqref{eq:W13} and \eqref{eq:Ct}.
  One has
\begin{align} \Sscr_W(h)= \left(
 \sum_{\substack{t_3\geqslant 1\\(t_3,W_3)=1\\p\mid t_3\Rightarrow p\equiv3\, (4)}}
 \frac{c_{t_3}(h)}{t_3^2}
 \right)
 \left(
 \sum_{\substack{t_1\geqslant 1\\(t_1,W_1)=1\\p\mid t_1\Rightarrow p\equiv1\, (4)}}
 \frac{c_{t_1}(h)}{t_1^2}C(t)^2
 \right).\label{eq:seriesequal}\end{align}
 Recall \eqref{eq:DvQ}. We also define
 \begin{align} \label{defineScalWh}\mathcal{S}_{W}(h)= \left(
 \sum_{\substack{t_3\geqslant 1\\(t_3,W_3)=1\\p\mid t_3\Rightarrow p\equiv3\, (4)}}
 \frac{c_{t_3}(h)}{t_3^2}
 \right)
 \left(
 \sum_{\substack{t_1\geqslant 1\\(t_1,W_1)=1\\p\mid t_1\Rightarrow p\equiv1\, (4)}}
 \frac{c_{t_1}(h)}{t_1^2}D_{v,Q}(t)^2
 \right).\end{align}
 \begin{lemma}\label{lem:forseries}Suppose that $0<|h|<(\log X)^{100}$.
   One has\begin{align*} \mathcal{S}_{W}(h)=
 \frac{64A^2\log v}{\pi^2g_1(W_1)^2} \Sscr_W(h)+O_{\eps,W}((\log X)^{\eps}).\end{align*}
 \end{lemma}
\begin{proof} In view of \eqref{eq:Ramanujan-bound2}, $\sigma_{-1}(h)\ll_\eps |h|^\eps\ll_{\eps} (\log X)^\eps$ and \eqref{eq:seriesequal}, it suffices to prove that
 \begin{align}
 \sum_{\substack{t\geqslant 1\\(t,W_1)=1\\p\mid t\Rightarrow p\equiv1\, (4)}}
 \frac{c_{t}(h)}{t^2}D_{v,W_1}(t)^2=&\,
 \frac{64A^2\log v}{\pi^2g_1(W_1)^2}
 \sum_{\substack{t\geqslant 1\\(t,W_1)=1\\p\mid t\Rightarrow p\equiv1\, (4)}}
 \frac{c_{t}(h)}{t^2}C(t)^2
 \notag
 \\ &\, +O_{\eps,W_1}((\log X)^{\eps}).\label{eq:DtoC}\end{align}
 By Lemma \ref{lemma:DvQt}, we have
  \begin{align*}
 \sum_{\substack{t\geqslant 1\\(t,W_1)=1\\p\mid t\Rightarrow p\equiv1\, (4)}}
 \frac{c_{t}(h)}{t^2}D_{v,W_1}(t)^2=
 \frac{64A^2\log v}{\pi^2g_1(W_1)^2}
 \sum_{\substack{t\geqslant 1\\(t,W_1)=1\\p\mid t\Rightarrow p\equiv1\, (4)}}
 \frac{c_{t}(h)}{t^2}C(t)^2+O_{W_1}(E),\end{align*}
 where
  \begin{align*}E=\sum_{t}\frac{c_t(h)}{t^{7/4}}+\frac{1}{\log v}\sum_{t}\frac{c_t(h)}{t^{3/2}}.\end{align*}
  This proves \eqref{eq:DtoC} by using \eqref{eq:Ramanujan-bound2}. The proof of this lemma is complete.
 \end{proof}

\begin{lemma}\label{lemma-shift}Let $(\alpha, W)=1$. Suppose that \[
  0<|h|<(\log X)^{100},
  \qquad 4W\mid h.
 \] 
  Uniformly for intervals $I\subset[X/3,3X]$ of length $\asymp X$ and uniformly for (the above) $h$, the following hold as $X\to\infty$. One has
 \begin{equation}
  \sum_{\substack{m\in I\\m\equiv\alpha\, (W)\\m\equiv1\, (4)}}
  \rho_v(m)\rho_v(m+h)
  =\frac{64A^2g_1(W_3)^2}{W\log v}|I|\Sscr_W(h)
  +o_{W}\!\left(\frac{X}{\log X}\right).
  \label{eq:rho-shifted}
 \end{equation}
\end{lemma}

\begin{proof}We only need to consider the case $h>0$, because when $h<0$, we can interchange the role of $\rho_v(m)$ and $\rho_v(m+h)$ by setting $m+h=m'$ and $m=m'+(-h)$. 

Expanding both copies of $\rho_v$ by \eqref{eq:rhov}, we deduce that the left hand side of \eqref{eq:rho-shifted} is equal to
\[
 \frac{1}{\log^2 v}
 \sum_{\substack{a,b\leqslant v\\(ab,W)=1\\p\mid ab\Rightarrow p\equiv1\, (4)}}
 \frac{\mu(a)\mu(b)\log (v/a) \log (v/b)}{g_2(a)g_2(b)}\sum_{\substack{m\in I\\m\equiv\alpha\, (W)\\m\equiv1\, (4)\\a\mid m\\b\mid m+h}}
 r_2(m)r_2(m+h).
\]
We have introduced the condition $(ab,W)=1$ in the above, because the inner sum is zero when $(ab,W)>1$. Applying Lemma \ref{lem:shifted-r2} with $c_1=a$ and $c_2=b$, 
 \begin{equation}
  \sum_{\substack{m\in I\\m\equiv\alpha\, (W)\\m\equiv1\, (4)}}
  \rho_v(m)\rho_v(m+h)
  =
  \mathcal{M}_2+ O(\mathcal{E}_2),
  \label{eq:M2E2}
 \end{equation}
 where
\begin{align*}
 \mathcal{M}_2=\frac{\pi^2 g_1(W)^2|I|}{W\log^2 v}
 \sum_{\substack{a,b\leqslant v\\(ab,W)=1\\p\mid ab\Rightarrow p\equiv1\, (4)}}&\,
 \frac{\mu(a)\mu(b)\log (v/a) \log (v/b)}{ab}
 \\ &\,\times\sum_{\substack{t\geqslant 1\\(t,2W)=1}}
 \frac{c_t(h)(a,t)(b,t)
 \chi((a^2,t))\chi((b^2,t))}
 {t^2\Psi(a,t)\Psi(b,t)}
\end{align*}
and
\begin{align*}
 \mathcal{E}_2= \frac{1}{\log^2 v}
 \sum_{\substack{a,b\leqslant v\\(ab,W)=1\\p\mid ab\Rightarrow p\equiv1\, (4)}}&\,
 \frac{|\mu(a)\mu(b)|\log (v/a) \log (v/b)}{g_2(a)g_2(b)}
 \\ &\,\times\left(W^{1/2}abx^{3/4+\eps}
 +(ab)^{1/2}x^{5/6+\eps}\right).
\end{align*}
We first deduce that
\begin{align}\label{boundE2}
 \mathcal{E}_2\ll_{\eps,W} X^{3/4+\eps}v^4
 +X^{5/6+\eps}v^3 =o_{W}\Big( \frac{X}{\log X}\Big).
\end{align}

Now we deal with the main term. Since $a$ and $b$ contain only primes congruent to $1$ modulo $4$, we have $\chi((a^2,t))\chi((b^2,t))=1$. Split $t=t_1t_3$, where the prime factors of $t_j$ are congruent to $j$ modulo $4$. Then one has $(a,t)=(a,t_1)$, $(b,t)=(b,t_1)$, $\Psi(a,t)=\Psi(a,t_1)$ and $\Psi(b,t)=\Psi(b,t_1)$. Therefore, we arrive at
\begin{equation*}\mathcal{M}_2=\frac{\pi^2 g_1(W)^2|I|}{W\log^2 v}
 \left(
 \sum_{\substack{t_3\geqslant 1\\(t_3,W_3)=1\\p\mid t_3\Rightarrow p\equiv3\, (4)}}
 \frac{c_{t_3}(h)}{t_3^2}
 \right)
 \left(
 \sum_{\substack{t_1\geqslant 1\\(t_1,W_1)=1\\p\mid t_1\Rightarrow p\equiv1\, (4)}}
 \frac{c_{t_1}(h)}{t_1^2}D_{v,W_1}(t_1)^2
 \right).
\end{equation*}
Recalling \eqref{defineScalWh}, we have
\begin{equation*}\mathcal{M}_2=\frac{\pi^2 g_1(W)^2|I|}{W\log^2 v} \mathcal{S}_{W}(h).
\end{equation*}
Note that $$\frac{\pi^2 g_1(W)^2|I|}{W\log^2 v}\cdot 
 \frac{64A^2\log v}{\pi^2g_1(W_1)^2} \Sscr_W(h)=\frac{64A^2g_1(W_3)^2}{W\log v}|I|\Sscr_W(h).$$
By Lemma \ref{lem:forseries},  \begin{equation}\label{boundM2}
  \mathcal{M}_2
  =\frac{64A^2g_1(W_3)^2}{W\log v}|I|\Sscr_W(h)
  +O_{\eps, W}\big(X(\log X)^{-2+\eps}\big).
 \end{equation}
We complete the proof by combining \eqref{eq:M2E2}, \eqref{boundE2} and \eqref{boundM2}..
\end{proof}

\begin{lemma}\label{lemma-secondmoment}Let $(\alpha, W)=1$. Uniformly for intervals $I\subset[X/3,3X]$ of length $\asymp X$, the following hold as $X\to\infty$. One has
 \begin{equation*}
  \sum_{\substack{m\in I\\m\equiv\alpha\, (W)\\m\equiv1\, (4)}}
  \rho_v(m)^2\ll_{W}\frac{X}{\sqrt{\log X}}.
 \end{equation*}
\end{lemma}

\begin{proof}It suffices to prove for $x\in I$ that
 \begin{equation}
  \sum_{\substack{m\leqslant x\\m\equiv\alpha\, (W)\\m\equiv1\, (4)}}
  \rho_v(m)^2\ll_{W}\frac{x}{\sqrt{\log x}}.
  \label{eq:rho-diagonal}
 \end{equation}
Expanding $\rho_v$ by \eqref{eq:rhov}, we deduce that the left hand side of \eqref{eq:rho-diagonal} is equal to
\[
 \frac{1}{\log^2 v}
 \sum_{\substack{a,b\leqslant v\\(ab,W)=1\\p\mid ab\Rightarrow p\equiv1\, (4)}}
 \frac{\mu(a)\mu(b)\log (v/a) \log (v/b)}{g_2(a)g_2(b)}\sum_{\substack{m\leqslant x\\m\equiv\alpha\, (W)\\m\equiv1\, (4)\\ [a,b]\mid m}}
 r_2(m)^2.
\]
Again, the condition $(ab,W)=1$ was introduced, because the inner sum is zero when $(ab,W)>1$.

Recalling $Z^{(1)}_{v,Q}$ and $Z^{(2)}_{v,Q}$ defined in Lemma \ref{lemma:XZonetwo} and applying Lemma \ref{lemma:second-r2}, we deduce that the left hand side of \eqref{eq:rho-diagonal} is further equal to
\begin{equation}\label{eq:containZ12}
 \frac{g_3(W)}{W}\cdot\frac{x}{(\log v)^2}
 \left((\log x+G_5(W))Z^{(1)}_{v,W_1}-2Z^{(2)}_{v,W_1}\right)+O_{\eps}(v^2WX^{3/4+\eps}).
\end{equation}
By \eqref{eq:Zonetwo}, the first part of \eqref{eq:containZ12} is $\ll_{W}\frac{x}{\sqrt{\log x}}$, which is also an upper bound of $v^2WX^{3/4+\eps}$.
 This proves \eqref{eq:rho-diagonal} and the proof of this lemma is complete.
\end{proof}

\section{Proof of Proposition~\ref{lem:variance}}\label{sec:branch-moments}

\subsection{Averaging the singular series}\label{sec:singular-average}

\begin{lemma}\label{lem:quadratic-Ramanujan}
Suppose that $(2W,t)=1$ and $1\leqslant t\leqslant (Y/W)^{1/3}$. Then
\begin{align*}\sum_{z_1,z_2\in \mathcal{Z}}c_t\Big(\frac{z_1^2-z_2^2}{\lambda_0}\Big)=|\mathcal{Z}|^2\phi(t)t^{-1}+O(\frac{Yt}{W}).\end{align*}
\end{lemma}

\begin{proof}
Since $(\lambda_0,t)=1$, we have $c_t\Big(\frac{z_1^2-z_2^2}{\lambda_0}\Big)=c_t(z_1^2-z_2^2)$. Therefore,
\begin{align*}\sum_{z_1,z_2\in \mathcal{Z}}c_t\Big(\frac{z_1^2-z_2^2}{\lambda_0}\Big)
=\sum_{a,b\pmod{t}}c_t(a^2-b^2)\sum_{\substack{z_1,z_2\in \mathcal{Z} \\ z_1\equiv a\pmod{t} \\ z_2\equiv b\pmod{t}}}1.\end{align*}
Note that $|\{z\in \mathcal{Z} :\ z\equiv a\pmod{t}\}|=\frac{|\mathcal{Z}|}{t}+O(1)=\frac{Y}{2W}+O(1)$. Then 
\begin{align*}\sum_{z_1,z_2\in \mathcal{Z}}c_t\Big(\frac{z_1^2-z_2^2}{\lambda_0}\Big)
=\frac{|\mathcal{Z}|^2}{t^2}\sum_{a,b\pmod{t}}c_t(a^2-b^2)+O(\frac{Yt}{W})+O(t^2).\end{align*}
We deduce that 
\begin{align*}\sum_{a,b\pmod{t}}c_t(a^2-b^2)=\sum_{\substack{x\pmod{t} \\ (x,t)=1}}\Big|\sum_{a\pmod{t}}e\big(\frac{xa^2}{t}\big)\Big|^2=\phi(t)t.\end{align*}
This proves the lemma.
\end{proof}

\begin{lemma}\label{prop:mean-singular-series}
Let $\lambda_0$, $W$ and $\calZ$ be as in \eqref{eq:definej0lambda0}, \eqref{eq:W-product} and \eqref{eq:definecalZ}, respectively. Then
\begin{equation}
 \frac1{|\mathcal{Z}|^2}\sum_{\substack{z_1,z_2\in\calZ\\z_1\ne z_2}}
 \Sscr_W\!\left(\frac{z_1^2-z_2^2}{\lambda_0}\right)
 =M_W+o_W(1),
 \label{eq:mean-singular-series}
\end{equation}
where $M_W$ is defined in \eqref{eq:MW}.
\end{lemma}
\begin{proof}Choose
\[
 T=(Y/W)^{1/3}.
\]
For $z_1\ne z_2 \in \mathcal{Z}$, one has $0<|(z_1^2-z_2^2)/\lambda_0|\leqslant Y^2$. Write $h_{z_1,z_2}=(z_1^2-z_2^2)/\lambda_0$. By \eqref{eq:Ramanujan-bounds} and \eqref{eq:gamma}, uniformly in $z_1,z_2$,
\begin{equation*}
 \sum_{t>T}\frac{\gamma(t)|c_t(h_{z_1,z_2})|}{t^2}
 \ll\frac{\tau(|h_{z_1,z_2}|)}T \ll W^{1/3}Y^{-1+o(1)}=o_W(1).
\end{equation*}
Thus,
\begin{equation}
\frac1{|\mathcal{Z}|^2}\sum_{\substack{z_1,z_2\in\calZ\\u\ne v}} \sum_{t>T}\frac{\gamma(t)|c_t(h_{z_1,z_2})|}{t^2}
 =o_W(1).
 \label{eq:tail}
\end{equation}

For $t\leqslant T$, we have $\sum_{z_1=z_2\in \mathcal{Z}}c_t(h_{z_1,z_2}) \ll \frac{Y}{W}T$. Then we deduce from Lemma \ref{lem:quadratic-Ramanujan} that
\begin{align*}
\frac1{|\mathcal{Z}|^2}\sum_{\substack{z_1,z_2\in\calZ\\z_1\ne z_2}} \sum_{\substack{t\leqslant T \\ (2W,t)=1}}\frac{\gamma(t)c_t(h_{z_1,z_2})}{t^2}=&\, 
\frac1{|\mathcal{Z}|^2}\sum_{\substack{t\leqslant T \\ (2W,t)=1}}\frac{\gamma(t)}{t^2}\sum_{\substack{z_1,z_2\in\calZ\\u\ne v}} c_t(h_{z_1,z_2})
\\ =&\, \frac1{|\mathcal{Z}|^2}\sum_{\substack{t\leqslant T \\ (2W,t)=1}}\frac{\gamma(t)}{t^2}\Big(
|\mathcal{Z}|^2\phi(t)t^{-1}+O(\frac{YT}{W})\Big)
\\ =&\, \sum_{\substack{t\leqslant T \\ (2W,t)=1}}\frac{\gamma(t)\phi(t)}{t^3}+O\big((Y/W)^{-2/3}\big)
\\ =&\, M_W+O\big((Y/W)^{-1/3}\big).
\end{align*}This in combination with \eqref{eq:tail} proves \eqref{eq:mean-singular-series}.
\end{proof}

\subsection{First and second moments}

Recall
\begin{equation*}
 \mu_0=\frac1{|\calN|}\sum_{n\in\calN}R(n).
\end{equation*}
\begin{lemma}\label{lem:mean}One has
\begin{equation}
 \mu_0=(16Ag_1(W_3)+o_W(1))\frac{|\mathcal{Z}|}{\sqrt{\log v}}.
 \label{eq:branch-mean}
\end{equation}
\end{lemma}

\begin{proof}We have
\begin{equation*}
 \sum_{n\in\calN}R(n)
 =\sum_{z\in\calZ}\sum_{n\in\calN}\rho_v(m_{n,z}).
\end{equation*}
Fix $z\in\calZ$. As $n$ runs over $\calN$, the integer $m:=m_{n,z}=(n-z^2)/\lambda_0$ runs through a fixed reduced residue class modulo $W$, is congruent to $1$ modulo $4$, and lies in an interval of length $X/\lambda_0$. Applying \eqref{eq:rho-first}, uniformly in $z$, we have
\[
 \sum_{n\in\calN}\rho_v(m_z)
=\sum_{m}\rho_v(m) =\frac{4Ag_1(W_3)}{W\sqrt{\log v}}\frac {X}{\lambda_0}
 +o_W\!\left(\frac{X}{\sqrt{\log X}}\right).
\]
Sum over $z$ and recall  \eqref{sizeNandZ}, we obtain
\begin{equation*}
 \mu_0=\frac{1}{|\mathcal{N}|}\sum_{n\in\calN}R(n)
 =(16Ag_1(W_3)+o_W(1))\frac{|\mathcal{Z}|}{\sqrt{\log v}}.
\end{equation*}
This proves
 \eqref{eq:branch-mean}.
\end{proof}

\begin{lemma}\label{lem:second-moment-branch}One has
\begin{equation}
 \sum_{n\in\calN}R(n)^2=\frac{64A^2g_1(W_3)^2}{W\log v}\frac{X}{\lambda_0} |\mathcal{Z}|^2M_W+o_W\Big(\frac{XY^2}{\log X}\Big)+O_W\Big(\frac{XY}{\sqrt{\log X}}\Big).
 \label{eq:offdiag-total}
\end{equation}
\end{lemma}

\begin{proof}Expand
\begin{equation}
 \sum_{n\in\calN}R(n)^2
 =\sum_{z_1,z_2\in\calZ}\sum_{n\in\calN}\rho_v(m_{n,z_1})\rho_v(m_{n,z_2}).
 \label{eq:second-expand}
\end{equation}
For $z_1=z_2$, Lemma \ref{lemma-secondmoment} gives a total contribution
\begin{equation}
\sum_{z\in\calZ}\sum_{n\in\calN}\rho_v(m_{n,z})^2 \ll_W\frac{XY}{\sqrt{\log X}}.
 \label{eq:diagonal-total}
\end{equation}

For $z_1\not=z_2$, write $m=m_{n,z_1}=\frac{n-z_1^2}{\lambda_0}$ and $h=m_{n,z_1}-m_{n,z_2}=\frac{z_1^2-z_2^2}{\lambda_0}$. Then $m\equiv 1\pmod{4}$ and $m\equiv \lambda_0^{-1}(c-\beta_c)\pmod{W}$. Moreover $4W\mid h$.
The interval in the $m$ variable has length $X/\lambda_0$, so by Lemma \ref{lemma-shift}
\begin{align}\sum_{n\in\calN}\rho_v(m_{n,z_1})\rho_v(m_{n,z_2})
 =&\,\sum_{m}\rho_v(m)\rho_v(m-h)\notag
\\ =&\, \frac{64A^2g_1(W_3)^2}{W\log v}\frac{X}{\lambda_0}\Sscr_W(h)
 +o_W\!\left(\frac X{\log X}\right),
 \label{eq:pair-branch}
\end{align}
uniformly in $z_1,z_2$. After summing over $z_1,z_2(z_1\not=z_2)$, the accumulated error from \eqref{eq:pair-branch} is $o_W(XY^2/\log X)$, and by Lemma \ref{prop:mean-singular-series}, the main term contribution from \eqref{eq:pair-branch} is
\begin{equation*}
 \frac{64A^2g_1(W_3)^2}{W\log v}\frac{X}{\lambda_0}  |\mathcal{Z}|^2(M_W+o_W(1)).
\end{equation*}
Combining this with \eqref{eq:second-expand} and \eqref{eq:diagonal-total}
proves \eqref{eq:offdiag-total}.
\end{proof}

\begin{proof}[Proof of Proposition~\ref{lem:variance}]
By Lemma \ref{lem:mean}, $\mu_0>0$ for sufficiently large $X$, since
$A>0$ and $g_1(W_3)>0$.
By \eqref{sizeNandZ} and \eqref{eq:branch-mean}, we have
\begin{equation}
 |\calN|\mu_0^2
 =(1+o_W(1))\frac{64A^2g_1(W_3)^2}{W\log v}\frac{X}{\lambda_0} |\mathcal{Z}|^2\asymp_W \frac{XY^2}{\log X}.
 \label{eq:mean-square-scale}
\end{equation}
It is important that the main term in \eqref{eq:offdiag-total} and \eqref{eq:mean-square-scale} is essential the same in view of \eqref{eq:MW-tail}.

Since $\mu_0$ is the exact mean,
\[
 \sum_{n\in\calN}|R(n)-\mu_0|^2
 =\sum_{n\in\calN}R(n)^2-|\calN|\mu_0^2.
\]
Combining Lemma \ref{lem:second-moment-branch} with \eqref{eq:mean-square-scale}, we obtain
\begin{align}
 \sum_{n\in\calN}\big|R(n)-\mu_0\big|^2=&\,\frac{64A^2g_1(W_3)^2}{W\log v}\frac{X}{\lambda_0} |\mathcal{Z}|^2\Big(M_W-1\Big)\notag
 \\ &\, +o_W\Big(\frac{XY^2}{\log X}\Big)+O_W\Big(\frac{XY}{\sqrt{\log X}}\Big).\label{eq:2ndmean}
\end{align}

By \eqref{eq:mean-square-scale} and \eqref{eq:2ndmean}, 
\begin{align*}\sum_{n\in\calN}|R(n)-\mu_0|^2=
 \Big(M_W-1+o_W(1)+O_W\Big(\frac{\sqrt{\log X}}{Y}\Big)\Big)|\calN|\mu_0^2.
\end{align*}
Applying \eqref{eq:MW-tail} proves \eqref{eq:branch-variance} and completes
the proof of Proposition \ref{lem:variance}.
\end{proof}

\section{Proof of Theorem~\ref{thm:critical-exceptions}}\label{sec:critical-exceptions}

\subsection{Sums of two squares in arithmetic progressions}
Let
\[
\Ssq=\{m\in\N:m=x^2+y^2\text{ for some }x,y\in\Z\}.
\]
A positive integer $m$ belongs to $\Ssq$ if and only if every prime $p\equiv3\pmod4$ occurs to an even exponent in the prime factorization of $m$. Recall the Landau--Ramanujan constant $A$ defined in \eqref{eq:LRconstant}.

Define $\lambda_8$ on $\Z/8\Z$ by
\begin{equation}\label{eq:pd-local-2}
\begin{array}{c|cccccccc}
c\pmod8&0&1&2&3&4&5&6&7\\ \hline
\lambda_8(c)&\frac18&\frac14&\frac14&0&\frac18&\frac14&0&0
\end{array}
\end{equation}
and for an odd prime $p\equiv3\pmod4$, define a function on $\Z/p^2\Z$ by
\begin{equation}\label{eq:pd-local-p}
\lambda_{p^2}(c)=
\begin{cases}
\dfrac{p+1}{p^3},&p\nmid c,\\[2mm]
0,&p\mid c\text{ but }p^2\nmid c,\\[2mm]
\dfrac1{p^2},&p^2\mid c.
\end{cases}
\end{equation}

Let $\mathcal P$ be a finite set of primes congruent to $3\pmod4$, and put
\begin{equation}\label{eq:pd-modulus}
Q=\prod_{p\in\mathcal P}p^2,
\qquad q=8Q.
\end{equation}
By the Chinese remainder theorem, every residue class $c\pmod q$ has components $c_8\pmod8$ and $c_p\pmod{p^2}$.  Set
\begin{equation}\label{eq:pd-lambda}
\lambda_q(c)=\lambda_8(c_8)\prod_{p\in\mathcal P}\lambda_{p^2}(c_p).
\end{equation}
We introduce
\begin{equation*}
B(x;q,c)=\#\{m\leqslant x:m\in\Ssq,\ m\equiv c\pmod q\}.
\end{equation*}
\begin{lemma}\label{lem:pd-ap}
For fixed $q$ of the form \eqref{eq:pd-modulus}, uniformly over the finitely many residue classes $c\pmod q$,
\begin{equation*}
B(x;q,c)
=\left(\lambda_q(c)+o_q(1)\right)\frac{Ax}{\sqrt{\log x}}.
\end{equation*}
Consequently,
\begin{equation*}
\#\{X<m\leqslant 2X:m\in\Ssq,\ m\equiv c\pmod q\}
=\left(\lambda_q(c)+o_q(1)\right)\frac{AX}{\sqrt{\log X}}.
\end{equation*}
\end{lemma}
Note that we do not have the condition $(c,q)=1$ in Lemma \ref{lem:pd-ap}. Prachar \cite{Prachar} proved the
following.

\begin{lemma}[Prachar]\label{lem:pd-prachar}
Assume that $(a,k)=1$. If $a\equiv1\pmod{(4,k)}$, then
\[
B(x;k,a)=
\left(\dfrac{(4,k)}{(2,k)k}\displaystyle\prod_{\substack{p\mid k \\ p\equiv 3\pmod{4}}}(1+\frac{1}{p})+o_k(1)\right)
\dfrac{Ax}{\sqrt{\log x}}.
\]
Otherwise, $B(x;k,a)=0$.
\end{lemma}
Taking $k=8r$ with all prime factors of $r$ congruent to $3$ modulo $4$, by Lemma \ref{lem:pd-prachar}, we have for $(a,8r)=1$ that
\begin{align}\label{eq:pd-reduced}B(x;8r,a)=
\begin{cases}
\left(\dfrac{1}{4r}\displaystyle\prod_{\substack{p\mid r }}(1+p^{-1})+o_r(1)\right)
\dfrac{Ax}{\sqrt{\log x}},&\ a\equiv1\pmod{4},\\[4mm]
0,&\ a\equiv3\pmod{4}.
\end{cases}
\end{align}

\begin{lemma}\label{lem:pd-odd}Assume that $(c,2)=1$. Then
\begin{align}\label{eq:pd-odd}
B(x;8Q,c)=
\begin{cases}\left(\frac{1}{4}\displaystyle\prod_{p\in\mathcal P}\lambda_{p^2}(c)+o_Q(1)\right)
\frac{Ax}{\sqrt{\log x}},&\ c\equiv1\pmod{4},\\[4mm]
0,&\ c\equiv3\pmod{4}.
\end{cases}
\end{align}
\end{lemma}

\begin{proof}It suffices to consider the case $c\equiv1\pmod{4}$.
If, for some $p\in\mathcal P$, one has $p\mid c$ but $p^2\nmid c$, then every $m\equiv c\pmod{p^2}$ satisfies $v_p(m)=1$, so no such $m$ belongs to $\Ssq$.  Both sides of \eqref{eq:pd-odd} are then zero at the main term level because $\lambda_{p^2}(c)=0$.

Assume now that, for every $p\in\mathcal P$, either $p\nmid c$ or $p^2\mid c$.  Split
\[
\mathcal P=\mathcal P_0\sqcup\mathcal P_1,
\]
where $p\in\mathcal P_0$ when $p^2\mid c$, and $p\in\mathcal P_1$ when $p\nmid c$.  Put
\[
D=\prod_{p\in\mathcal P_0}p^2,
\qquad R=\prod_{p\in\mathcal P_1}p^2.
\]
Note that $D\mid c$, $(R,c)=1$, $Q=DR$, and $D\equiv1\pmod8$. For $m\in\Ssq$ with $m\equiv c\pmod Q$, every $p\in\mathcal P_0$ occurs in $m$ to an even exponent at least two.  Thus $m=D\ell$ with $\ell\in\Ssq$.  The congruence $D\ell\equiv c\pmod{8Q}$ is equivalent to $\ell\equiv c/D\pmod{8R}$.  We apply \eqref{eq:pd-reduced} with $r=R$, at height $x/D$, to obtain
\[
B(x;8Q,c)=\left(\frac{1}{4Q}\prod_{p\in\mathcal P_1}(1+p^{-1})+o_Q(1)\right)
\frac{Ax}{\sqrt{\log x}}.
\]
Recall \eqref{eq:pd-local-p}. The above coefficient is exactly
\[
\frac{1}{4Q}\prod_{p\in\mathcal P_1}(1+p^{-1})=\frac{1}{4}\prod_{p\in\mathcal P_0}\frac1{p^2}
\prod_{p\in\mathcal P_1}\frac{p+1}{p^3}
=\frac{1}{4}\prod_{p\in\mathcal P}\lambda_{p^2}(c),
\]
and this completes the proof.
\end{proof}

\begin{proof}[Proof of Lemma \ref{lem:pd-ap}] In view of Lemma \ref{lem:pd-odd}, it remains only to explain the eight entries of $\lambda_8$. Since $a^2+b^2\not\equiv 3,6,7\pmod{8}$, this coincides with $\lambda_8(3)=\lambda_8(6)=\lambda_8(7)=0$.

For every integer $j\geqslant0$, multiplication by $2^j$ preserves membership
in $\Ssq$ and satisfies $\lambda_{p^2}(2^jc)=\lambda_{p^2}(c)$ for $p\in\mathcal P$.

If $c\equiv1$ or $5\pmod8$, then the desired conclusion follows immediately from \eqref{eq:pd-local-2} and Lemma \ref{lem:pd-odd}.

If $c\equiv2\pmod8$, write $m=2\ell$ and $c=2c'$.  Then $\ell$ is odd and $c'\equiv 1\pmod{4}$. Moreover, $m\equiv c\pmod{8Q}$ is equivalent to $\ell\equiv c'\pmod{4Q}$. In other words, $\ell\equiv c',c'+4Q\pmod{8Q}$. Applying Lemma \ref{lem:pd-odd} to $\{\ell \leqslant x/2:\ell\in\Ssq,\ \ell\equiv c''\pmod{8Q}\}$ with $c''=c', c'+4Q$, we obtain
$$B(x;q,c)=\left(\frac{1}{4}\displaystyle\prod_{p\in\mathcal P}\lambda_{p^2}(c)+o_Q(1)\right)
\frac{Ax}{\sqrt{\log x}}.$$

If $c\equiv4\pmod8$, write $m=4\ell$ and $c=4c'$.  Then both $\ell$ and $c'$ are odd. Moreover, $m\equiv c\pmod{8Q}$ is equivalent to $\ell\equiv c'\pmod{2Q}$. In other words, $\ell\equiv c',c'+2Q,c'+4Q,c'+6Q\pmod{8Q}$. Two of $\{c',c'+2Q,c'+4Q,c'+6Q\}$ are congruent to $3$ modulo $4$. Applying Lemma \ref{lem:pd-odd} to $\{\ell \leqslant x/4:\ell\in\Ssq,\ \ell\equiv c''\pmod{8Q}\}$ and summing over the two nonzero cases, we obtain the factor $\frac{1}{8}$, which coincides with $\lambda_8(4)$.

If $c\equiv0\pmod8$, write $m=8\ell$ and $c=8c'$. Then
$m\in\Ssq$ if and only if $\ell\in\Ssq$, and the congruence
$m\equiv c\pmod{8Q}$ is equivalent to $\ell\equiv c'\pmod Q$.
If some $p\in\mathcal P$ satisfies $p\mid c'$ but $p^2\nmid c'$,
both the count and the predicted main term vanish. Otherwise, put
\[
D=\prod_{\substack{p\in\mathcal P\\p^2\mid c'}}p^2,
\qquad R=Q/D.
\]
Writing $\ell=Du$ gives
\[
B(x;8Q,c)=B\!\left(\frac{x}{8D};R,\frac{c'}D\right),
\qquad \left(\frac{c'}D,R\right)=1.
\]
Lemma~\ref{lem:pd-prachar} now yields
\[
B(x;8Q,c)=
\left(\frac{1}{8Q}\prod_{p\mid R}(1+p^{-1})+o_Q(1)\right)
\frac{Ax}{\sqrt{\log x}}.
\]
The coefficient equals $\frac18\prod_{p\in\mathcal P}\lambda_{p^2}(c)$,
which agrees with $\lambda_8(0)=\frac18$.
\end{proof}

For $d=8$, $d=p^2$ with $p\equiv3\pmod4$ prime, or $d=q$, define
\[
\Sigma_d(a)=\sum_{r\pmod d}\lambda_d(a-r^2).
\]
The Chinese remainder theorem and the product formula \eqref{eq:pd-lambda} imply
\begin{equation}\label{eq:pd-sigma-factor}
\Sigma_q(a)=\Sigma_8(a)\prod_{p\in\mathcal P}\Sigma_{p^2}(a).
\end{equation}

\begin{lemma}\label{lem:pd-local-factors}
(i)
Let $p\equiv3\pmod4$ be prime.  Then
\begin{equation*}
\Sigma_{p^2}(1)=1-\frac1p.
\end{equation*}
(ii) For $a_0\in\{1,2,3,5,6\}$,
\begin{equation}\label{eq:pd-two-factor}
\Sigma_8(a_0)=
\begin{cases}
1,&a_0\equiv3\pmod8,\\[1mm]
\dfrac32,&a_0\equiv1,2,5,6\pmod8.
\end{cases}
\end{equation}
\end{lemma}

\begin{proof}
(i) Note that 
\begin{align*}&\#\{r\pmod{p^2}:p^2\mid 1-r^2\}=2,
\\ 
& \#\{r\pmod{p^2}:p\nmid 1-r^2\}=p^2-2p,
\\ & \#\{r\pmod{p^2}:p\| 1-r^2\}=2(p-1).\end{align*}
Thus,
\begin{align*}
\Sigma_{p^2}(1)
=2\frac1{p^2}+(p^2-2p)\frac{p+1}{p^3}=1-\frac1p.
\end{align*}
(ii)
Modulo $8$, the square $r^2$ equals $0$ for two residue classes $r$, equals $4$ for two residue classes, and equals $1$ for four residue classes. Thus \eqref{eq:pd-two-factor} follows from the table \eqref{eq:pd-local-2} and the identity
\begin{equation*}
\Sigma_8(a)=2\lambda_8(a)+2\lambda_8(a-4)+4\lambda_8(a-1).
\qedhere
\end{equation*}
This proof of the lemma is complete.
\end{proof}

\subsection{Completion of the proof}

\begin{proof}[Proof of Theorem \ref{thm:critical-exceptions}]
For fixed $C>0$, we choose $T=T_C$ depending only on $C$ such that
\begin{align}\label{eq:pd-choose-modulus}\frac{3AC}{2}\prod_{\substack{p\leqslant T \\ p\equiv 3\pmod{4}}}\left(1-\frac1p\right)<\frac{1}{4}.\end{align}
This is possible because $\sum_{p\equiv 3\pmod{4}}p^{-1}=\infty$.
Let 
$$q=8Q,\ \ \ Q=\prod_{\substack{p\leqslant T \\ p\equiv 3\pmod{4}}}p^2.$$
Fix an admissible residue $a_0\in\{1,2,3,5,6\}$.  By the Chinese remainder theorem, choose $a\pmod q$ satisfying
\begin{equation*}
a\equiv a_0\pmod8,
\qquad a\equiv1\pmod{Q}.
\end{equation*}
Every integer $n\equiv a\pmod q$ lies in the prescribed admissible class $a_0\pmod8$.

Let
\begin{equation*}
Y=\left\lfloor C\sqrt{\log(2X)}\right\rfloor.
\end{equation*}
If $X<n\le2X$ and $|z|\leqslant C\sqrt{\log n}$, then $|z|\leqslant Y$.  Since $z^2=(-z)^2$, it is enough to consider the $Y+1$ values
\[
z=0,1,\ldots,Y.
\]

Let $R_a(X)$ be the number of integers $n$ satisfying
\[
X<n\leqslant 2X,
\qquad n\equiv a\pmod q,
\]
that possess at least one representation
\[
n=x^2+y^2+z^2,
\qquad |z|\leqslant C\sqrt{\log n}.
\]
For every such $n$, at least one $z\in[0,Y]$ satisfies $n-z^2\in\Ssq$.  Therefore the union bound gives
\begin{equation*}
R_a(X)
\leqslant
\sum_{0\leqslant z\leqslant Y}
\#\{X<n\leqslant 2X:n\equiv a\pmod q,\ n-z^2\in\Ssq\}.
\end{equation*}

For a fixed $z$, put $m=n-z^2$.  Then
\[
m\equiv a-z^2\pmod q.
\]
The interval for $m$ is $(X-z^2,2X-z^2]$.  Replacing it by $(X,2X]$ changes the count by at most $O(z^2+1)$.
Since $z^2\ll_C\log X=o_C(X/\sqrt{\log X})$ and $q$ depends only on $C$,
Lemma~\ref{lem:pd-ap} gives, uniformly for $0\leqslant z\leqslant Y$,
\begin{align*}
&\,\#\{X<n\leqslant 2X:n\equiv a\pmod q,\ n-z^2\in\Ssq\}
\\= &\,\lambda_q(a-z^2)\frac{AX}{\sqrt{\log X}}+o_C(X/\sqrt{\log X}).
\end{align*}
Therefore,
\begin{align}\label{eq:pd-first-moment}
R_a(X)
&\leqslant \frac{AX}{\sqrt{\log X}}
\sum_{0\leqslant z\leqslant Y}\lambda_q(a-z^2)+o_C(X).
\end{align}

Note that $\lambda_q(a-z^2)$ is periodic modulo $q$ as a function of $z$.  Hence
\begin{equation}\label{eq:pd-periodic}
\sum_{0\leqslant z\leqslant Y}\lambda_q(a-z^2)
=\frac{Y+1}{q}\Sigma_q(a)+O_q(1).
\end{equation}
Since $(Y+1)/\sqrt{\log X}=C+o(1)$, equations \eqref{eq:pd-first-moment} and \eqref{eq:pd-periodic} yield
\begin{equation*}
R_a(X)
\leqslant \left(AC\Sigma_q(a)+o_C(1)\right)\frac{X}{q}.
\end{equation*}
It follows from \eqref{eq:pd-sigma-factor} and  Lemma \ref{lem:pd-local-factors} that
\begin{equation*}
\Sigma_q(a)\leqslant \frac{3}{2}
\prod_{\substack{p\leqslant T \\ p\equiv 3\pmod{4}}}\left(1-\frac1p\right).
\end{equation*}
Recalling \eqref{eq:pd-choose-modulus}, we obtain $AC\Sigma_q(a)<\frac14$.
Thus, for all sufficiently large $X$,
\begin{equation}\label{eq:pd-represented}
R_a(X)\leqslant \frac13\frac{X}{q}.
\end{equation}
On the other hand,
\begin{equation}\label{eq:pd-progression-size}
\#\{X<n\leqslant 2X:n\equiv a\pmod q\}
=\frac{X}{q}+O(1).
\end{equation}
Subtracting \eqref{eq:pd-represented} from \eqref{eq:pd-progression-size}, we find 
\[
\#\mathcal E_{C,a_0}(X)\geqslant \frac12\frac{X}{q}
\]
for all sufficiently large $X$.  In particular, Theorem \ref{thm:critical-exceptions} holds with $\delta_C=\frac1{2q}$.
\end{proof}

\section{The sum of a squarefree number and a power of $2$}

\begin{theorem}[Granville-Soundararajan] Assume that there are $\leqslant 2\log x/(\log\log x)^2$ primes $p\leqslant x$ for which $p^2\mid 2^{p-1}-1$. Then for all but $O(x/\log x)$ of odd integers $n\leqslant x$ can be represented as the sum of a squarefree number and a power of $2$.
\end{theorem}

Granville and Soundararajan \cite{GS} pointed out that it would be nicer to have an ``~if and only if~" statement of some kind.

For $q$ odd, we use $ord_{q}(2)$ to denote the order of $2$ in $(\Z/q\Z)^\ast$. Following \cite{GS}, we write 
$$\omega(p)=ord_{p^2}(2).$$

\begin{theorem}\label{thm:squarefree-covering} The following two are equivalent.

(i) Almost all odd numbers can be represented as the sum of a squarefree number and a power of $2$;

(ii) There do not exist an integer $m\geqslant1$, pairwise distinct odd primes
$p_1,\ldots,p_m$, and integers $a_1,\ldots,a_m$ such that the residue classes
$a_j\pmod{\omega(p_j)}$, $1\leqslant j\leqslant m$, cover $\mathbb{Z}$.
  
\end{theorem}

Before proving Theorem~\ref{thm:squarefree-covering}, we need the following lemma.

\begin{lemma}\label{lem:squarefree-local}
Assume condition (ii) of Theorem~\ref{thm:squarefree-covering}.
Let $W$ be as in \eqref{eq:W-product} and put $\sigma=\ord_{W^2}(2)$.
Then, for every $c\pmod{W^2}$, there exists an integer
$1\leqslant\gamma_c\leqslant\sigma$ such that
$p^2\nmid c-2^{\gamma_c}$ for every prime $p\mid W$.
\end{lemma}

\begin{proof}
Fix $c\pmod{W^2}$ and suppose that no such $\gamma_c$ exists.
Then, for each $1\leqslant\gamma\leqslant\sigma$, there is a prime
$q_\gamma\mid W$ such that $q_\gamma^2\mid c-2^\gamma$.
For each prime $q$ arising in this way, choose the unique integer
$1\leqslant b_q\leqslant\omega(q)$ such that
$c\equiv2^{b_q}\pmod{q^2}$.
It follows that $\gamma\equiv b_{q_\gamma}\pmod{\omega(q_\gamma)}$.

Thus the residue classes $b_{q_\gamma}\pmod{\omega(q_\gamma)}$,
$1\leqslant\gamma\leqslant\sigma$, cover all integers from $1$ to $\sigma$.
Since $\omega(q_\gamma)\mid\sigma$ for every $\gamma$, they cover
$\mathbb{Z}$ by periodicity.
The residue $b_q$ depends only on $q$, so deleting repetitions gives a
covering system $a_j\pmod{\omega(p_j)}$, $1\leqslant j\leqslant m$,
with pairwise distinct odd primes $p_1,\ldots,p_m$.
This contradicts condition (ii).
\end{proof}

\begin{proof}[Proof of Theorem~\ref{thm:squarefree-covering}]
Granville and Soundararajan \cite[Theorem 5]{GS} have proved that (i) implies (ii).
It remains to prove that (ii) implies (i). Assume (ii).
Let 
\begin{align}\label{choosew}w=\log\log X.\end{align}
Let $W$ be as in \eqref{eq:W-product} and put $\sigma=\ord_{W^2}(2)$.
Fix $c\pmod{W^2}$ and choose $\gamma=\gamma_c$ as in
Lemma~\ref{lem:squarefree-local}.
Let $X$ be sufficiently large.

Set $$T=\lfloor\frac{X}{\log X}\rfloor.$$
Now let 
$$\mathcal{X}=\{T< n\leqslant X:\ n\equiv c\pmod{W^2},\ n\equiv 1\pmod{2}\}$$
and
$$\mathcal{L}=\{1\leqslant k\leqslant \sigma L:\ k\equiv \gamma\pmod{\sigma}\},$$
where 
$$L=\lfloor \frac{\log_2 T}{\sigma}\rfloor.$$
Note that $2^k\leqslant T\leqslant n$ for all $k\in \mathcal{L}$ and $n\in \mathcal{X}$.

Let 
$$\rho(n)=\#\{ k\in \mathcal{L}:\ \mu^2(n-2^k)=1 \}.$$
We first prove
\begin{align}\sum_{n\in \mathcal{X}}|\rho(n)-L|\ll \frac{1}{w}\cdot\frac{XL}{W^2}+L\sqrt{X}.\label{aim}
\end{align}
We deduce that 
\begin{align*}\rho(n)=\sum_{k\in \mathcal{L}}\mu^2(n-2^k)=\sum_{d<\sqrt{X}}\mu(d)\sum_{\substack{k\in \mathcal{L} \\ 2^k\equiv n\pmod{d^2}}}1.
\end{align*}

Note that $2^k\equiv n\pmod{d^2}$ implies that $d$ is odd. We first consider the case that there exists a prime $p\mid W$ such that $p\mid d$.
According to the choose of $c$ and $\gamma$, one has $2^k-n\equiv 2^{\gamma}-c\not\equiv 0\pmod{p^2}$. Thus 
$$\sum_{\substack{k\in \mathcal{L} \\ 2^k\equiv n\pmod{d^2}}}1=0$$
for $(d,2W)>1$. The contribution from $d=1$ is $L$.
Now we obtain
\begin{align*}\rho(n)=\sum_{k\in \mathcal{L}}\mu^2(n-2^k)=L+\sum_{\substack{1<d<\sqrt{X} \\ (d,2W)=1}}\mu(d)\sum_{\substack{k\in \mathcal{L} \\ 2^k\equiv n\pmod{d^2}}}1.
\end{align*}
Then we have
\begin{align*}\sum_{n\in \mathcal{X}}|\rho(n)-L|=
&\,\sum_{n\in \mathcal{X}}\Big|\sum_{\substack{1<d<\sqrt{X} \\ (d,2W)=1}}\mu(d)\sum_{\substack{k\in \mathcal{L} \\ 2^k\equiv n\pmod{d^2}}}1\Big|
\\ \leqslant
&\, \sum_{\substack{1<d<\sqrt{X} \\ (d,2W)=1}}\mu^2(d)\sum_{\substack{k\in \mathcal{L}}}\sum_{\substack{n\in \mathcal{X} \\ n\equiv 2^k\pmod{d^2}}}1.
\\ \ll &\, \sum_{\substack{1<d<\sqrt{X} \\ (d,2W)=1}}\mu^2(d)\sum_{\substack{k\in \mathcal{L}}}\Big(\frac{X}{2d^2W^2}+O(1)\Big)
\\ \ll &\,  \frac{1}{w}\cdot\frac{XL}{W^2}+O(L\sqrt{X}).
\end{align*}
This proves \eqref{aim}.

Let $\mathcal{E}(T,X;c,W^2)$ be the set consisting of odd $n\in [T,X]$ satisfying $n\equiv c\pmod{W^2}$ and that $n$ cannot be written as the sum of a squarefree number and a power of $2$.  Let $\mathcal{E}(X)$ be the set consisting of odd $n\in [1,X]$ satisfying that $n$ cannot be written as the sum of a squarefree number and a power of $2$. Then we have
\begin{align}\label{eq:exceprelation}
\#\mathcal{E}(X)\leqslant T+W^2\,\#\mathcal{E}(T,X;c,W^2).\end{align}
For $n\in \mathcal{E}(T,X;c,W^2)$, one has $\rho(n)=0$ and thus
$$L\,\#\mathcal{E}(T,X;c,W^2) \leqslant \sum_{n\in \mathcal{X}}|\rho(n)-L|.$$
Then by \eqref{aim}, we obtain
\begin{align}\label{eq:forexceprelation}
W^2\,\#\mathcal{E}(T,X;c,W^2)\ll\frac{X}{w}+W^2\sqrt{X}.\end{align}
Recalling \eqref{eq:W-product} and \eqref{choosew}, one has $W\ll (\log X)^2$. Then by \eqref{eq:exceprelation} and \eqref{eq:forexceprelation}, we conclude that
$$\#\mathcal{E}(X)\ll \frac{X}{\log\log X}.$$
This completes the proof of Theorem~\ref{thm:squarefree-covering}.
\end{proof}

\section*{Acknowledgments}
This work is supported by the National Key Research and Development Program of China (Grant No. 2021YFA1000700) and the National Natural Science Foundation of China (Grant No. 12471088).

\end{document}